\documentclass[11pt,onecolumn]{article}

\usepackage[a4paper,margin=1.15in]{geometry}
\usepackage{lmodern}
\usepackage{microtype}
\usepackage[dvipsnames]{xcolor}
\usepackage{amsmath}
\usepackage{amssymb}
\usepackage{amsthm}
\usepackage{aliascnt}
\usepackage{graphicx}
\usepackage{mathtools}
\usepackage{authblk}
\usepackage{enumitem}
\usepackage{titlesec}
\usepackage{fancyhdr}
\usepackage{hyperref}
\usepackage{cleveref}

\definecolor{linkblue}{HTML}{1F4E79}
\hypersetup{
    colorlinks=true,
    linkcolor=linkblue,
    citecolor=linkblue,
    urlcolor=linkblue
}

\titleformat{\section}{\large\bfseries}{\thesection}{0.75em}{}
\titleformat{\subsection}{\normalsize\bfseries}{\thesubsection}{0.75em}{}
\titlespacing*{\section}{0pt}{2.3ex plus 0.8ex minus 0.2ex}{1.1ex plus 0.2ex}
\titlespacing*{\subsection}{0pt}{1.8ex plus 0.6ex minus 0.2ex}{0.8ex plus 0.2ex}

\setlist[itemize]{leftmargin=2.2em,itemsep=0.25ex,topsep=0.6ex}
\usepackage{tikz-cd}

\theoremstyle{plain}
\newtheorem{theorem}{Theorem}[section]

\newaliascnt{corollary}{theorem}
\newtheorem{corollary}[corollary]{Corollary}
\aliascntresetthe{corollary}

\newaliascnt{proposition}{theorem}
\newtheorem{proposition}[proposition]{Proposition}
\aliascntresetthe{proposition}

\newaliascnt{lemma}{theorem}
\newtheorem{lemma}[lemma]{Lemma}
\aliascntresetthe{lemma}

\theoremstyle{definition}
\newaliascnt{definition}{theorem}
\newtheorem{definition}[definition]{Definition}
\aliascntresetthe{definition}

\newaliascnt{example}{theorem}

\aliascntresetthe{example}

\newaliascnt{remark}{theorem}
\newtheorem{remark}[remark]{Remark}
\aliascntresetthe{remark}

\crefname{theorem}{theorem}{theorems}
\Crefname{theorem}{Theorem}{Theorems}
\crefname{corollary}{corollary}{corollaries}
\Crefname{corollary}{Corollary}{Corollaries}
\crefname{proposition}{proposition}{propositions}
\Crefname{proposition}{Proposition}{Propositions}
\crefname{lemma}{lemma}{lemmas}
\Crefname{lemma}{Lemma}{Lemmas}
\crefname{definition}{definition}{definitions}
\Crefname{definition}{Definition}{Definitions}
\crefname{example}{example}{examples}
\Crefname{example}{Example}{Examples}
\crefname{remark}{remark}{remarks}
\Crefname{remark}{Remark}{Remarks}

\newcommand{\C}{\mathbb{C}}

\newcommand{\littletaller}{\mathchoice{\vphantom{\big|}}{}{}{}}
\newcommand\restr[2]{{
		\left.\kern-\nulldelimiterspace
		#1 
		\littletaller 
		\right|_{#2} 
}}

\title{Operator Systems in Duality}
\author[]{Markus Dannemüller}
\author[]{Tim Netzer}
\affil[]{Department of Mathematics, University of Innsbruck}
\date{\today}

\begin{document}

\maketitle

\begin{abstract}
The usual dual of an infinite-dimensional order-unit space need not
carry an order unit, and this is the main reason why duality for operator systems is notoriously hard.
Rather than trying to compute a canonical dual object, we introduce a relative
notion of duality for operator systems:\ Two systems are in duality when a pairing of their underlying vector spaces induces
conic pairings at every matrix level.  We develop basic
properties of this notion, showing that beyond finite-dimensional operator systems, also separable
$B(H)$ and $C^*$-algebras with a faithful trace admit duals, while some other operator systems do not.  We prove that
subsystems, quotients and suitable tensor products of systems with duals again admit duals. As
consequences, we show that finite-level maximality is dual to finite-level minimality, and
we establish weakly continuous and weak$^*$ closed realizations into $\ell^\infty$-products of
matrix algebras.
\end{abstract}

%\tableofcontents

\section{Introduction}

Duality is one of the organizing principles of mathematics.  It allows one to
replace objects by spaces of tests on them, to turn geometric questions into
separation problems, and to translate constructions on one side into sometimes
simpler or more revealing constructions on the other.  In functional analysis,
convexity, and operator algebra, duality is especially powerful because
positivity is frequently detected by positive functionals.  Understanding the
right dual object therefore often determines
which order, topological, and representation-theoretic tools are available.

Operator systems provide the intrinsic, order-theoretic language for
unital self-adjoint subspaces of $C^*$-algebras.  Since the Choi--Effros
representation theorem, they have become a basic framework for studying
matrix-ordered structures and completely positive maps \cite{ChoiEffros1977,Paulsen2002}.
In finite
dimensions, duality is easy to understand, since  the vector-space dual of an operator
system has canonical matrix cones, and the resulting dual operator system is indeed a
standard tool in finite-dimensional operator systems theory.

In infinite dimensions, however, computing duals of operator systems is
notoriously problematic: though the dual cone is always a natural positive
cone, it may fail to admit a distinguished order unit making the dual space
into another order-unit space, and hence it may fail to underlie an operator
system.  For example, the Banach dual of $C(X)$ is the space of signed
measures on $X$, and the positive cone of this dual does in general not admit an
order unit.

There are several ways to deal with this obstruction in the literature.  One
approach is to enlarge the category. In \cite{Ng2011}  self-adjoint
subspaces of $B(H)$ with their induced matrix norms and matrix cones are studied,
without requiring the existence of an order unit. In later work \cite{Ng2022}, a dualizability condition is developed, involving a change of the dual norm to an equivalent matrix norm that makes the system a good non-unital dual.  
A related recent approach is developed in \cite{FritzNetzer2026Modules}, where
regularly ordered Banach spaces form the enriching category for a
module-theoretic treatment of nonunital operator systems and their duals.
A different line of work uses unitalization constructions for nonunital ordered
spaces \cite{Werner2004,vanDobben2026} or approximately unital operator
systems  \cite{Huang2011}.  
All these theories show that
the missing order unit can be repaired, but only by adding extra structure,
changing norms, or passing to a larger nonunital  category.

In this paper we take a different approach. Instead of computing \emph{the} dual of an operator system, we define what it means for two operator systems to be in duality.
This notion, which is well-established for classical spaces and cones, keeps the dual object more flexible. The extension to operator systems is quite straightforward, and it turns out that surprisingly many operator systems \emph{do} admit at least one dual.

We  first  show that  basic examples such as $B(H)$ for separable $H$ and  $C^*$-algebras with a
faithful trace admit duals; they are actually even self-dual as operator systems. We next prove
results about classical constructions: Subsystems, quotients, and suitable tensor products of systems with a dual will again have duals. This shows for example that all systems with a separable $C^*$-envelope have duals. But we also provide examples of operator systems that do not admit any dual in our sense.

We then study consequences of  duality. We prove  that finite-level maximality is dual to finite-level minimality, extending results from \cite{DannemuellerNetzer2026Subhomogeneous} to infinite dimensions. We further show that
systems with a dual have a weakly continuous, and under additional assumptions also a weak$^*$ closed realization in
$\ell^\infty$-products of matrix algebras. 

\section{Preliminaries}

\subsection{Duality of spaces and convex cones}
\label{sec:space}

A good reference for the concepts in this section is \cite{Barvinok2002}.

Let $V$ and $W$ be real vector spaces.  A \emph{duality}, or \emph{dual pairing},
between $V$ and $W$ is a bilinear form
\[
    \langle \cdot,\cdot\rangle \colon V\times W\to \mathbb{R}
\]
which separates points on both sides: For every nonzero $v\in V$ there is
some $w\in W$ with $\langle v,w\rangle\neq 0$, and for every nonzero
$w\in W$ there is some $v\in V$ with $\langle v,w\rangle\neq 0$. 
When such a pairing is fixed, we say that $V$ and $W$ are in duality.  

The standard example is the duality between a real vector space $V$ and its
algebraic dual space $V^*$, the space of all real-linear functionals
$\varphi\colon V\to\mathbb{R}$.  The pairing is evaluation:
\[
    \langle v,\varphi\rangle=\varphi(v),
    \qquad v\in V,\ \varphi\in V^*.
\]

In finite dimensions, every duality is of this form up to the natural
identification of one space with the dual of the other.  Indeed, a duality
between  $V$ and $W$ induces injective linear maps
\[
    V\hookrightarrow W^*,\qquad v\mapsto \langle v,\cdot\rangle,
    \qquad\text{and}\qquad
    W\hookrightarrow V^*,\qquad w\mapsto \langle \cdot,w\rangle .
\]
Hence $\dim V\leqslant \dim W$ and $\dim W\leqslant \dim V$, so $\dim V=\dim W$ and
both maps are isomorphisms.  After identifying $W$ with $V^*$ by the second
isomorphism, the original pairing becomes precisely the evaluation pairing.
Thus, in finite dimensions, the canonical pairing of a space with its
algebraic dual is the only example, up to isomorphism of dual pairs.

In infinite dimensions this uniqueness fails.  
For example, let
$X$ be a compact Hausdorff space and let $\mu$ be a positive regular Borel
measure on $X$ with full support, meaning that every nonempty open subset of
$X$ has positive measure.  Then the real vector space $C(X,\mathbb R)$ of continuous
real-valued functions on $X$ is in duality with itself under the pairing
\[
    \langle f,g\rangle=\int_X f(x)g(x)\,d\mu(x),
    \qquad f,g\in C(X,\mathbb R).
\]
The full-support assumption guarantees separation: if $f\neq 0$, then there
is a nonempty open set on which $f$ is nonzero, and hence
$\langle f,f\rangle=\int_X f(x)^2\,d\mu(x)>0$. This is not
the full algebraic duality between $C(X,\mathbb R)$ and its algebraic dual; we only obtain those functionals of the
form
\[
    f\longmapsto \int_X f(x)g(x)\,d\mu(x)
\]
for some $g\in C(X,\mathbb R)$. It is not even the topological duality when $C(X,\mathbb R)$ is considered a Banach space; here the dual space consists of all signed Borel measures on $X$, and not all of them need to have a continuous density with respect to $\mu.$

A duality also determines a natural weak topology.  If $V$ and $W$ are in
duality, the weak topology on $V$ induced by $W$ is
the coarsest vector-space topology on $V$ for which all maps
\[
    V\to\mathbb{R},\qquad v\mapsto \langle v,w\rangle,
    \qquad w\in W,
\]
are continuous. 
Because the pairing separates points, this topology is Hausdorff.  Its
continuous linear functionals are exactly those realized by the pairing: If
$\lambda\colon V\to\mathbb{R}$ is linear and weakly continuous,
then there is a unique $w\in W$ such that
\[
    \lambda(v)=\langle v,w\rangle \qquad\text{for all } v\in V.
\]
Similarly, the pairing gives a weak topology on $W$, whose continuous linear functionals are exactly the
maps $w\mapsto \langle v,w\rangle$ with $v\in V$.

This also gives the right notion of dual, or adjoint, of a continuous linear
map.  Suppose $(V_1,W_1)$ and $(V_2,W_2)$ are dual pairs, each equipped with
its weak topology, and let $\psi\colon V_1\to V_2$ be linear.  If $\psi$ is weakly 
continuous, then for every
$w_2\in W_2$ the functional
\[
    v_1\longmapsto \langle \psi(v_1),w_2\rangle
\]
is weakly continuous on $V_1$.  Since all such continuous
linear functionals are represented by elements of $W_1$, there is a unique
element $\psi^*(w_2)\in W_1$ such that
\[
    \langle \psi(v_1),w_2\rangle=\langle v_1,\psi^*(w_2)\rangle
    \qquad\text{for all }v_1\in V_1.
\]
The assignment $w_2\mapsto \psi^*(w_2)$ is linear, and defines the dual map
\(
    \psi^*\colon W_2\to W_1.
\)

Pairings also give a natural duality for convex cones.  Suppose
$C\subseteq V$ is a convex cone.  Its dual cone is
\[
    C^\vee=\{w\in W\mid \langle v,w\rangle\geqslant 0
    \text{ for all } v\in C\}.
\]
Thus $C^\vee$ consists of those linear functionals on $V$ which are realized
by elements of $W$ and are nonnegative on $C$.  Similarly, if
$D\subseteq W$ is a convex cone, its dual cone in $V$ is
\[
    D^\vee\coloneqq\{v\in V\mid \langle v,w\rangle\geqslant 0
    \text{ for all } w\in D\}.
\]
The double dual cone is the weak closure of the original cone:
\[
    (C^\vee)^\vee=\overline{C}.
\]
If $C\subseteq V$ and $D\subseteq W$ are convex cones, the pairing between $V$ and $W$ is called a \emph{conic pairing} when
\[
    D=C^\vee
    \qquad\text{and}\qquad
    C=D^\vee.
\]
So membership in one of the cones is exactly determined by positivity at each positive functional given by elements from the dual. This in particular implies that both cones are weakly closed.

Conic pairings are compatible with dual maps.  Suppose
$(V_1,W_1)$ and $(V_2,W_2)$ are dual pairs with cones
$C_i\subseteq V_i$ and $D_i\subseteq W_i$ satisfying $D_i=C_i^\vee$ and
$C_i=D_i^\vee$.  A linear map $\psi\colon V_1\to V_2$ is called \emph{positive} if
$\psi(C_1)\subseteq C_2$.  Now if $\psi$ is weakly continuous, then $\psi$ is positive if and only if its dual map
$\psi^*$ is positive.

\subsection{Duality of operator systems}

Most concepts from this section can also be found in \cite{Paulsen2002}, for example.

An \emph{Archimedean order unit space} is a real vector space $V$, together with a salient convex cone $C$ and a distinguished point $e\in C,$ called an \emph{Archimedean order unit}, with the following two properties: First, for every $v\in V$ there exists some $r>0$ such that $re+v\in C$ (order unit property). Second, whenever $\varepsilon e+v\in C$ for all $\varepsilon >0,$ then $v\in C$ (Archimedean property).
The associated \emph{order-unit norm} is
\[
    \|v\|_{\rm ou}=\inf\{r>0\mid re\pm v\in C\},
\]
and the topology induced by this norm is called the \emph{order topology}.

We start with some easy preparatory lemmas.

\begin{lemma}
\label{lem:ordertopweak}
    Let $(V,C,e_V)$ and $(W,D,e_W)$ be Archimedean order unit spaces in conic
    duality.  Then the order topology on $V$ is finer than the weak topology
    induced by $W$.
\end{lemma}
\begin{proof}
    It is enough to show that, for each $w\in W$, the functional
    $v\mapsto\langle v,w\rangle$ is continuous for the order-unit norm on
    $V$.
    Since $W$ has an order unit, every element in $W$ is a difference of two elements from $D$. So we can restrict to the case $w\in D$. 
    
    Now take $v\in V$, choose $r>0$ with $re_V\pm v\in C,$ and obtain from duality
    \[
        0\leqslant \langle re_V\pm v,w\rangle
        =r\langle e_V,w\rangle\pm\langle v,w\rangle.
    \]
    This implies $\vert\langle v,w\rangle\vert\leqslant r\langle e_V,w\rangle,$ and taking the infimum over all such $r$ gives
    \[
        |\langle v,w\rangle|
        \leqslant \langle e_V,w\rangle\,\Vert v\Vert_{\rm ou},
    \]
    so the functional is order-norm continuous. 
\end{proof}

\begin{lemma}
\label{lem:dualsalient}
    Assume $C\subseteq V$ and $D\subseteq W$ are in conic duality, and that $C$ admits an order unit. Then $D$ is salient.
\end{lemma}
\begin{proof}
    Let $e\in C$ be an order unit for $C$, and suppose that
    $d\in D\cap(-D)$.  Since $D=C^\vee$, for every $c\in C$ we have both
    $\langle c,d\rangle\geqslant 0$ and
    $\langle c,-d\rangle\geqslant 0$.  Hence
    \[
        \langle c,d\rangle=0
        \qquad\text{for all }c\in C.
    \]
    Since $C$ has an order unit, every element of $V$ is a difference of two elements from $C$, and thus $d$ pairs trivially with every element from $V$.
    The separating property of
    the duality implies $d=0$. 
\end{proof}

An \emph{operator system} is a noncommutative extension of an Archimedean order unit space.  It consists of a sequence of Archimedean order unit spaces $({\rm Her}_m(V),C_m,I_m\otimes e)$, compatible with  scalar matrix conjugation: If $x\in {\rm Mat}_{m,n}(\mathbb{C})$
and $a\in C_m$, then
\[
    x^*ax\in C_n.
\]
Here we define ${\rm Her}_m(V)={\rm Her}_m(\mathbb C)\otimes V$ and use matrix conjugation on the first tensor factor.

The concrete model of an operator system is a unital complete order embedding
$V\hookrightarrow B(H)_{\rm sa}$, where $H$ is a Hilbert space and $B(H)_{\rm sa}$ is the  space of self-adjoint bounded operators on $H$.  In this case $e$ is the preimage of the identity operator, and
$C_m$ is the cone inherited from the positive operators in
$B(H^{\oplus m})$.  The Choi--Effros theorem says that every
operator system is of this form.

Given two  operator systems with cones
$C_m\subseteq {\rm Her}_m(V)$ and $D_m\subseteq {\rm Her}_m(W)$ and order units $e_V,e_W$, a linear map $\psi\colon V\to W$ is called \emph{unital} if $\psi(e_V)=e_W$. It is called 
\emph{m-positive} if
\[
  {\rm id}_{{\rm Her}_m(\mathbb C)}\otimes \psi\colon {\rm Her}_m(V)\to{\rm Her}_m(W)  
\]
is positive with respect to $C_m$ and $D_m$, and \emph{completely positive} if it is \(m\)-positive for all \(m\).

Any duality of vector spaces \[
    \langle\cdot,\cdot\rangle\colon V\times W\to\mathbb R
\]
extends to a duality of matrix spaces ${\rm Her}_m(V)$ and ${\rm Her}_m(W)$, by combining with the trace-inner product on matrices. These pairings are compatible with scalar matrix conjugation:
\[
    \langle x^*ax,b\rangle=\langle a,xbx^*\rangle
\]
for $a\in{\rm Her}_m(V), b\in{\rm Her}_n(W)$ and $x\in{\rm Mat}_{m,n}(\mathbb C).$ In particular, scalar matrix conjugation is weakly continuous.

\begin{lemma}
\label{lem:dualconj}
    Assume a family of convex cones $C_m\subseteq{\rm Her}_m(V)$ is closed under scalar matrix conjugation. Then for any given duality between $V$ and $W$, the family of levelwise duals $C_m^\vee\subseteq{\rm Her}_m(W)$ is also closed under scalar matrix conjugation. 
\end{lemma}
\begin{proof}
    Let $b\in C_m^\vee$ and let $x\in {\rm Mat}_{m,n}(\mathbb C)$.  We have to
    show that $x^*bx\in C_n^\vee$.  Thus let $a\in C_n$.  By compatibility of
    the matrix pairings with scalar conjugation,
    \[
        \langle a,x^*bx\rangle=\langle xax^*,b\rangle.
    \]
    From $xax^*\in C_m$ and $b\in C_m^\vee$ it follows that
    $\langle xax^*,b\rangle\geqslant 0.$
    This holds for every $a\in C_n$, so $x^*bx\in C_n^\vee$.
\end{proof}

The central notion for what follows is the following form of duality for operator systems:
\begin{definition}
    Two  operator systems are in \emph{operator system duality}, if there is a duality of the underlying vector spaces, whose induced pairing is a conic pairing at every matrix level, i.e.\ for which the cones $C_m$ and $D_m$ are dual to each other. 
\end{definition}

Note that the definition is symmetric.  It does not single out one system as
\emph{the} dual of the other; instead, the matrix cones on both sides are required
to be exactly the dual cones of the cones on the opposite side. 

Given two such pairings, for each weakly continuous linear map
$\psi\colon V_1\to V_2$, $\psi$ is
completely positive if and only if  $\psi^*$ is completely positive.

\section{Existence of dual operator systems}

In this section we examine when an operator system admits a dual in the sense defined above.  Finite-dimensional systems behave as expected, but there are also genuinely infinite-dimensional and even self-dual examples.

\subsection{Finite-dimensional operator systems}

Finite-dimensional operator systems always admit duals. This is well-known since \cite{ChoiEffros1977}. Indeed, let
$\mathcal C$ be an operator system on the finite-dimensional space $V$, with convex cones $C_m\subseteq{\rm Her}_m(V).$ Put
$W=V^*$ and pair $V$ with $W$ by evaluation.  For each $m$, define
\[
    D_m\coloneqq C_m^\vee
    \subseteq {\rm Her}_m(V^*)
\]
with respect to the induced trace pairing between ${\rm Her}_m(V)$ and
${\rm Her}_m(V^*)$.

Since the cones $C_m$ are salient, their dual cones have nonempty interior.
Choosing an interior point of $D_1$ as order unit gives the usual dual
operator system. Thus the matrix cones of the dual are canonical, whereas its
order unit requires an additional choice.

\subsection{\texorpdfstring{$B(H)$ for a  separable Hilbert space}{B(H) for a Separable Hilbert space}}

The next example is the basic noncommutative infinite-dimensional one.  The
pairing is not the Banach-space duality of $B(H)$, but a weaker trace-class
pairing obtained from a positive trace-class operator with dense range.

\begin{theorem}
\label{thm:bh}
     Let $H$ be a separable Hilbert space. Then, as an operator system, $B(H)$ is in duality with itself.
\end{theorem}
\begin{proof}
   Consider a positive trace-class operator $\rho$ on $H$ with dense range. 
   Then $B(H)$ is in operator system duality with itself under the pairing
    \[
        \langle a,b\rangle
        \coloneqq{\rm Tr}(\rho a\rho b),
        \qquad a,b\in B(H)_{\rm sa}.
    \]
     If $a,b\geqslant 0$, then cyclicity of the trace gives
     \[
        {\rm Tr}(\rho a\rho b)
        ={\rm Tr}(b^{1/2}\rho a\rho b^{1/2})
        ={\rm Tr}\bigl((a^{1/2}\rho b^{1/2})^*
        (a^{1/2}\rho b^{1/2})\bigr)\geqslant 0.
     \]
    Conversely, if $a$ is not positive, choose $h\in H$ such that $\langle ah,h\rangle<0$.  
    Since the range of $\rho$ is dense, we may choose $x\in H$ with $\langle a\rho x,\rho x\rangle<0$. 
    For the rank-one positive operator $b=xx^*$ we get
    \[
        \langle a,b\rangle={\rm Tr}(\rho a\rho b)   =\langle a\rho x,\rho x\rangle<0.
    \]
    Thus the positive cone of $B(H)_{\rm sa}$ is self-dual for this pairing. The same argument also shows that the pairing separates points.
    Applying the same argument to $H^{\oplus m}$, with $I_m\otimes\rho$ in place of $\rho$, gives the required conic duality at every matrix level.
\end{proof}

\subsection{\texorpdfstring{$C^*$-algebras}{C*-algebras} with a faithful tracial state}

\begin{theorem}
\label{thm:trace}
    Let $(A,\tau)$ be a unital $C^*$-algebra with a faithful tracial state. Then, as an operator system, $A$ is in duality with itself.
\end{theorem}
\begin{proof}
    We use the real vector space $A_{\rm sa}$ and the pairing
    \[
        \langle a,b\rangle=\tau(ab),
        \qquad a,b\in A_{\rm sa}.
    \]
    This pairing separates points: If $a\in A_{\rm sa}$ is nonzero, then
    $a^2$ is positive and nonzero, so faithfulness gives
    $\langle a,a\rangle=\tau(a^2)>0$.

    At matrix level $m$, identify ${\rm Her}_m(A_{\rm sa})$ with the
    self-adjoint part of ${\rm Mat}_m(A)$ and use the faithful trace
    \[
        \tau_m([a_{ij}])=\sum_{i=1}^m \tau(a_{ii}).
    \]
    The induced pairing is
    \[
        \langle x,y\rangle=\tau_m(xy),
        \qquad x,y\in {\rm Her}_m(A_{\rm sa}).
    \]
    If $x,y\geqslant 0$ in ${\rm Mat}_m(A)$, then
    \[
        \tau_m(xy)=\tau_m(y^{1/2}xy^{1/2})\geqslant 0,
    \]
    because $\tau_m$ is tracial and $y^{1/2}xy^{1/2}$ is positive.

    Conversely, suppose $x\in {\rm Her}_m(A_{\rm sa})$ is not positive.  Let
    $x=x_+-x_-$ be its positive and negative parts in the $C^*$-algebra
    ${\rm Mat}_m(A)$.  Thus $x_-$ is positive and nonzero.  Taking $y=x_-$ gives
    \[
        \langle x,y\rangle
        =\tau_m(xx_-)
        =-\tau_m(x_-^2)<0,
    \]
    since $x_+x_-=0$ and $\tau_m$ is faithful.  Hence any element pairing
    nonnegatively with all positive elements must itself be positive. This proves the claim.
\end{proof}

An important class of examples is given by matrix-valued continuous
functions.  Let $X$ be a compact Hausdorff space, let $\mu$ be a Borel probability measure with full support on $X$, and let $d\geqslant 1$.  Then
${\rm Mat}_d(C(X))$ is a unital $C^*$-algebra, naturally identified with
$C(X,{\rm Mat}_d(\mathbb C))$.  The functional
\[
    \tau(f)=\int_X {\rm tr}(f(x))\,d\mu(x)
\]
where ${\rm tr}$ is the normalized matrix trace, is a faithful tracial
state.  Hence ${\rm Mat}_d(C(X))$ is in operator system duality with itself.

Another important class is provided by reduced group $C^*$-algebras.  If
$G$ is a discrete group and $C^*_r(G)$ is generated by the left regular
representation $\lambda\colon \mathbb C[G]\to B(\ell^2(G))$, then
\[
    \tau(a)=\langle a\delta_e,\delta_e\rangle
\]
defines a faithful tracial state on $C^*_r(G)$. Therefore
$C^*_r(G)$ is also in operator system duality with itself.

Irrational rotation algebras give another interesting noncommutative example \cite{Davidson1996}. For
irrational $\theta\in\mathbb R$, let $A_\theta$ be the universal unital
$C^*$-algebra generated by unitaries $u$ and $v$ satisfying
\[
    vu=e^{2\pi i\theta}uv.
\]
The canonical trace is determined on the dense $*$-subalgebra of trigonometric
polynomials by
\[
    \tau(u^m v^n)=
    \begin{cases}
        1, & (m,n)=(0,0),\\
        0, & \text{otherwise}.
    \end{cases}
\]
For irrational $\theta$ this trace is faithful.  Hence $A_\theta$ is in
operator system duality with itself by the theorem.

\subsection{Subsystems and quotients}

We next establish how duality behaves with respect to important constructions with operator systems. The first result addresses realizations and quotients.

Assume  $V$ and $W$ form a dual pair, and let
$\iota\colon U\hookrightarrow V$ be an injective linear map.  Define
\[
    U^\perp:=\{w\in W\mid \langle\iota(u),w\rangle=0
    \text{ for all }u\in U\}.
\]
Then $U^\perp$ is a weakly closed subspace of $W$ with canonical projection map $\pi\colon W\to W/U^\perp,$ and there is an induced pairing between
$U$ and $W/U^\perp$ given by
\[
    \langle u,\pi(w)\rangle:=\langle\iota(u),w\rangle.
\]
The pairings make both $\iota$ and the projection
$\pi\colon W\twoheadrightarrow W/U^\perp$ weakly continuous and dual to each other.
Similarly, let $\pi\colon W\twoheadrightarrow X$ be a surjective
linear map with $K=\ker(\pi)$ weakly closed. Set
\[
    K^\perp\coloneqq\{v\in V\mid \langle v,k\rangle=0
    \text{ for all }k\in K\}\subseteq V.
\]
For $v\in K^\perp$ and $w\in W$, define
\[
    \langle v,\pi(w)\rangle:=\langle v,w\rangle.
\]
This is well defined and  separates
points of $K^\perp$, as well as points of $X$ since \( K=(K^\perp)^\perp\) holds by weak closedness of $K$. Now the inclusion $\iota\colon K^\perp\hookrightarrow V$ is the dual to $\pi,$ just as before.

Note that the weak topologies on $U=K^\perp$ and $X=W/U^\perp$  are the subspace and quotient topologies of $V$ and $W,$ respectively.

\begin{lemma}
\label{lem:dual}
    Let $\langle\cdot,\cdot\rangle$ be a conic pairing between $C\subseteq V$ and $D\subseteq W$.  Set $B:=\iota^{-1}(C)\subseteq U$ as well as $E\coloneqq\overline{\pi(D)},$ where the closure is taken in the weak topology on $W/U^\perp$ induced by $U$. Then $B$ and $E$ are dual to each other, with respect to the duality between $U$ and $W/U^\perp$ from above.
\end{lemma}
\begin{proof}
    Since $D=C^\vee$, every element of $\pi(D)$ is nonnegative on $B$: If $u\in B$ and $d\in D$, then $\iota(u)\in C$, and hence
    \[
        \langle u,\pi(d)\rangle=\langle\iota(u),d\rangle\geqslant 0.
    \]
    Thus $\pi(D)\subseteq B^\vee$.  The cone $B^\vee$ is weakly closed, so
    \(E\subseteq B^\vee\).

    Conversely, let $x\in W/U^\perp$ satisfy $x\notin E$.
    By the separation theorem for dual pairs, there is some $u\in U$ such that
    \[
        \langle u,x\rangle<0 \qquad\text{and}\qquad \langle u,\pi(d)\rangle\geqslant 0 \text{ for all } d\in D.
    \]
    The second condition says
    \(\langle\iota(u),d\rangle\geqslant 0\) for all $d\in D$.  Since $C=D^\vee$, this implies $\iota(u)\in C$, hence $u\in B$.  But then from 
    \(\langle u,x\rangle<0\) we obtain $x\notin B^\vee$.  
    Therefore $B^\vee\subseteq E$, proving equality. This also implies $E^\vee=B,$ since $C$ and thus $B$ are weakly closed.
\end{proof}

\begin{theorem}
\label{thm:real}
    Assume the operator systems $\mathcal C, \mathcal D$ with cones $C_m\subseteq{\rm Her}_m(V), D_m\subseteq{\rm Her}_m(W)$ and Archimedean order units $e_V,e_W$ are in duality. Then for any injective linear map $\iota\colon U\hookrightarrow V$ with $e_V\in{\rm im}(\iota)$, the system $\mathcal B$ on $U$ defined by
    \[
        B_m\coloneqq ({\rm id}\otimes \iota)^{-1}(C_m)
    \]    
    admits a dual. Similarly, if $\pi\colon W\twoheadrightarrow X$ is a surjective linear map with weakly closed kernel and $e_V\in{\rm ker}(\pi)^\perp$, the system $\mathcal E$ on $X$ defined by 
    \[
      E_m=\overline{({\rm id}\otimes\pi)(D_m)},  
    \] 
    where the closure is taken in the quotient topology on $X$, admits a dual.
\end{theorem}
\begin{proof}
    First note that $\mathcal B$, as defined in the theorem, is actually an operator system on $U$, with Archimedean order unit $\iota^{-1}(e_V)$. We now use the induced duality between $U$ and $W/U^\perp$, as well as
    the quotient map $\pi\colon W\to W/U^\perp$. We set 
    \[
        E_m\coloneqq\overline{({\rm id}\otimes\pi)(D_m)}\subseteq {\rm Her}_m(W/U^\perp)
    \]       
    and obtain from \Cref{lem:dual} the duality between $\mathcal B$ and $\mathcal E$ levelwise. It remains to show that $\mathcal E$ is really an operator system. But as the image of a cone with order unit under a surjective linear map, $E_m$ admits the order unit $I_m\otimes\pi(e_W)$. By \Cref{lem:dualsalient}, $E_m$ is salient, and by \Cref{lem:dualconj}, the family is compatible with scalar matrix conjugations. Finally, the Archimedean property follows from weak closedness of $E_m$, which is a dual cone. Thus $\mathcal B$ is in duality with the operator system $\mathcal E$.

    For the second statement, put $U=\ker(\pi)^\perp\subseteq V$ and give
    $U$ the subsystem structure inherited from $\mathcal C$.  The assumption
    $e_V\in U$ makes this a unital subsystem. Applying the first part to the
    inclusion $U\hookrightarrow V$ identifies its dual cones with
    \[
        \overline{({\rm id}\otimes\pi)(D_m)}=E_m,
    \]
    so $\mathcal E$ is in duality with that subsystem and therefore admits a
    dual.
\end{proof}

\begin{remark}
    The first statement of \Cref{thm:real} generalizes
    \cite[Theorem~5.1]{DannemuellerNetzer2026Subhomogeneous} to infinite
    dimensions. Indeed, in the finite-dimensional setting of that result, a
    $d$-subhomogeneous operator system is realized as a subsystem of
    ${\rm Mat}_d(C(X)),$ where $X$ is a compact subset of a finite dimensional Euclidean space. Choosing a full-support measure on $X$ makes this
    ambient system self-dual by \Cref{thm:trace}, and \Cref{thm:real} then
    identifies the dual of the subsystem with a quotient of ${\rm Mat}_d(C(X))$.
    
    More generally, this argument only requires a realization of the
    subhomogeneous system as a subsystem of ${\rm Mat}_d(C(X))$ for a compact
    Hausdorff space $X$ admitting a Borel probability measure with full
    support (the existence of such a measure is not automatic for arbitrary
    compact Hausdorff spaces).  It is automatic if the operator system is
    separable.  Indeed, starting from any realization in ${\rm Mat}_d(C(X))$,
    choose a countable norm-dense subset and identify points of $X$ on which
    all matrix entries of its elements agree.  The resulting quotient $Y$ is
    compact and metrizable, every element of the system factors through $Y$,
    and the system is therefore realized in ${\rm Mat}_d(C(Y))$.  Since $Y$
    is separable, it admits a Borel probability measure with full support.
\end{remark}

\begin{corollary}
    Every operator system whose $C^*$-envelope is separable, or admits a
    faithful trace, has a dual.
\end{corollary}

\begin{proof}
    Every separable $C^*$-algebra admits a faithful representation on a separable Hilbert space \cite[Theorem~7.10]{Conway2000}. So the statement follows from 
    \Cref{thm:real,thm:bh,thm:trace}.
\end{proof}

\subsection{Tensor products}

Tensor products provide another construction where our  definition behaves
naturally. 

\begin{lemma}
\label{lem:psd}
    Assume $\mathcal C$ and $\mathcal D$ are operator systems in duality. Then for $c=\sum_i a_i\otimes v_i\in C_m$ and $d=\sum_j b_j\otimes w_j\in D_n,$ the matrix
    \[
      \sum_{i,j} \langle v_i,w_j\rangle\cdot a_i\otimes b_j^{\mathsf T}
      \in {\rm Her}_{mn}(\mathbb C)
    \]
    is positive semidefinite.
\end{lemma}
\begin{proof}
    Set $a=\sum_{i,j} \langle v_i,w_j\rangle\cdot
    a_i\otimes b_j^{\mathsf T}$, take $x\in\mathbb C^{mn}$, and write
    $x=\operatorname{vec}(y)$ for $y\in {\rm Mat}_{m,n}(\mathbb C)$, using
    the vectorization convention for which
    \[
        \operatorname{vec}(y)^*(a_i\otimes b_j^{\mathsf T})
        \operatorname{vec}(y)={\rm tr}(y^*a_iyb_j).
    \]
    Since $c\in C_m$ and the cones of
    $\mathcal C$ are closed under scalar matrix conjugation,
    \[
        y^*cy=\sum_i y^*a_iy\otimes v_i\in C_n.
    \]
    Since $d\in D_n=C_n^\vee$, we obtain
    \[
        0\leqslant \langle y^*cy,d\rangle
        =\sum_{i,j}\langle v_i,w_j\rangle\,
          {\rm tr}(y^*a_iyb_j)=x^*ax.
    \]
    So $a$ is indeed positive semidefinite.
\end{proof}

We now consider tensor products of operator systems, as defined and studied in \cite{KavrukPaulsenTodorovTomforde2011}. However, note that we will use the notions \emph{minimal} and \emph{maximal} in the opposite way, to match with the usual use for convex cones.

Let $V$ and $W$, as well as $X$ and $Y$, be spaces in duality.  On the algebraic
tensor products $V\otimes X$, $W\otimes Y$ we use the pairing defined by
\[
    \langle v\otimes x,w\otimes y\rangle
    =\langle v,w\rangle\langle x,y\rangle.
\]
Now further assume $\mathcal C$ and $\mathcal D,$ as well as $\mathcal E$ and $\mathcal F,$ are operator systems in duality on these spaces. We define the tensor product $\mathcal C\tilde\otimes_{\min}\mathcal E$ at level $m$ 
as the weak closure of the cone generated by all elements
\[
  x^*(c\otimes e)x, \qquad c\in C_s, e\in E_t, x\in{\rm Mat}_{st,m}(\C).
\]
Note that this is precisely the levelwise weak closure of the maximal tensor product from \cite{KavrukPaulsenTodorovTomforde2011}.
We then define  $\mathcal D\tilde\otimes_{\max}\mathcal F$ as the levelwise dual of $\mathcal C\tilde\otimes_{\min}\mathcal E.$ 

\begin{proposition}
\label{lem:minmax}
We have $\mathcal C\tilde\otimes_{\min}\mathcal E\ \subseteq\ \mathcal C\tilde\otimes_{\max}\mathcal E.$
\end{proposition}
\begin{proof}
    By definition of the maximal tensor product as the levelwise dual of the
    minimal tensor product of the duals, it is enough to check that
    \[
        \langle c\otimes e, x^*(d\otimes f)x\rangle\geqslant 0
    \]
    holds for all $c\in C_s$, $e\in E_t$, $d\in D_p$, $f\in F_q$, and
    $x\in {\rm Mat}_{pq,st}(\mathbb C).$
    Write
    \[
        c=\sum_i a_i\otimes v_i,
        \qquad d=\sum_j b_j\otimes w_j,
        \qquad
        e=\sum_k c_k\otimes x_k,
        \qquad f=\sum_l d_l\otimes y_l .
    \]
    Applying \Cref{lem:psd} to the dual pair $\mathcal C,\mathcal D$ gives
    the positive scalar matrix
    \[
        a=\sum_{i,j}\langle v_i,w_j\rangle\cdot
        a_i\otimes b_j^{\mathsf T}
        \in {\rm Her}_{sp}(\mathbb C),
    \]
    and applying the same lemma to $\mathcal E,\mathcal F$ gives the positive matrix
    \[
        b=\sum_{k,l}\langle x_k,y_l\rangle\cdot
        c_k\otimes d_l^{\mathsf T}
        \in {\rm Her}_{tq}(\mathbb C).
    \]
    Hence $a\otimes b$ is positive. Permuting the tensor factors, taking the
    transpose, and flipping the two resulting blocks preserves positivity.
    By the vectorization identity from \Cref{lem:psd}, the quadratic form of
    the resulting positive matrix at the vector corresponding to $x$ is
    \[
        \langle c\otimes e, x^*(d\otimes f)x\rangle.
    \]
    This pairing is therefore nonnegative, which proves the required
    inequality.
\end{proof}

\begin{theorem}
    $\mathcal C\tilde\otimes_{\min}\mathcal E$ and $\mathcal D\tilde\otimes_{\max}\mathcal F$ are operator systems in duality.
\end{theorem}
\begin{proof}
    It was shown in \cite{KavrukPaulsenTodorovTomforde2011} that their maximal tensor product is an operator system. Its levelwise weak closure remains compatible with scalar conjugation, has the same order unit, and is Archimedean because it is weakly closed. Thus for our minimal tensor product it remains only to check that the cones are salient. But since $\mathcal D\tilde\otimes_{\min}\mathcal F$ also has an order unit, its dual $\mathcal C\tilde\otimes_{\max}\mathcal E$ has salient cones by \Cref{lem:dualsalient}.
    From \Cref{lem:minmax} we get that also the cones of $\mathcal C\tilde\otimes_{\min}\mathcal E$ are salient. So all of our tensor products are indeed operator systems. Duality of  $\mathcal C\tilde\otimes_{\min}\mathcal E$ and $\mathcal D\tilde\otimes_{\max}\mathcal F$ is true by construction.
\end{proof}

\subsection{Non-uniqueness of the dual}
Duals need not be unique up to complete order isomorphism.  Let
 $\lambda$ be the Lebesgue measure on $[0,1]$.  The $C^*$-algebra
$C([0,1])$ is in duality (as an operator system) both with $C([0,1])$ and with
$L^\infty([0,1],\lambda)$, using in each case the pairing
\[
    \langle f,g\rangle=\int_0^1 f(x)g(x)\,d\lambda(x).
\]
They are not completely order isomorphic, since a unital complete order isomorphism is an isometry for the order-unit norm, while
$C([0,1])$ is separable and $L^\infty([0,1],\lambda)$ is not.

Since these are already $C^*$-algebras, this shows that also the  $C^*$-envelope of a dual system is not determined by the envelope of the primal system.

Also, duality of operator systems does not imply duality of the $C^*$-envelopes.
Let $C$ be a finite-dimensional
polyhedral cone which is not a simplex cone, and let $C^\vee$ be its dual
cone.  For instance, one may take the cone over a square.  The maximal
operator system over $C$ is dual, under the canonical finite-dimensional
pairing, to the minimal operator system over $C^\vee$.
By \cite{FritzNetzerThom2017}, the maximal operator system over a
closed cone has a finite-dimensional realization exactly when the cone is
polyhedral, while, among polyhedral cones, the minimal operator system has a
finite-dimensional realization exactly when the cone is a simplex cone
\cite{FritzNetzerThom2017}.  Thus the maximal system over $C$ has a
finite-dimensional realization, and therefore its $C^*$-envelope is
finite-dimensional.  On the other hand, since $C^\vee$ is again polyhedral
and not a simplex cone, the minimal system over $C^\vee$ has no
finite-dimensional realization.  Its $C^*$-envelope is consequently
infinite-dimensional.  Hence these two dual operator systems have $C^*$-envelopes of
different algebraic dimensions, and those envelopes cannot form a
dual pair.

\subsection{Systems without a dual}

\begin{lemma}
\label{lem:state}
    Assume the order-unit cones $C\subseteq V$ and $D\subseteq W$ are in
    conic duality, with order units $e_V$ and $e_W$, respectively. Then $C$
    admits a faithful state.
\end{lemma}
\begin{proof}
    Let $e=e_W$. Define
    \[
        \varphi(v)=\langle v,e\rangle,
        \qquad v\in V.
    \]
    Since $D=C^\vee$, this functional is nonnegative on $C$. We first claim
    that it is faithful. Let $c\in C$ and suppose that $\varphi(c)=0$. For an
    arbitrary $w\in W$, the order-unit property gives some $r>0$ such that
    $re\pm w\in D$.  Pairing these two elements with $c$
    gives
    \[
        0\leqslant \langle c,e\rangle\pm\langle c,w\rangle=\pm\langle c,w\rangle.
    \]
    Hence $\langle c,w\rangle=0$ for
    every $w\in W$, and duality implies $c=0$. In particular,
    $\varphi(e_V)>0$, where $e_V$ is the order unit of $C$. Dividing
    $\varphi$ by $\varphi(e_V)$ therefore gives a faithful state.
\end{proof}

\begin{corollary}
    If $H$ is a non-separable Hilbert space, then $B(H)$ admits no dual
    operator system in our sense.
\end{corollary}
\begin{proof}
    It is well-known that the $C^*$-algebra $B(H)$ admits no faithful state, whereas
    \Cref{lem:state} shows that the existence of a dual operator system would
    imply the existence of one.
\end{proof}

\section{Consequences of duality}

Having established existence results and compatibility with constructions, we now turn to consequences.  The first one explains how finite-level testing on one side becomes finite-level generation on the dual side.

\subsection{Maximality vs.\ minimality}

Fix $d\geqslant 1$.  We call an operator system $\mathcal C$ on $V$
$d$-maximal if, for every $m$ and every $a\in{\rm Her}_m(V)$,
\[
    a\in C_m
    \quad\Longleftrightarrow\quad
    x^*ax\in C_d \mbox{ for all }x\in{\rm Mat}_{m,d}(\mathbb C).
\]
If $\mathcal C$ and $\mathcal D$ are dual to each other, we call
$\mathcal D$ $d$-minimal if each $D_m$ equals the weak closure of the cone
generated by all elements
\[
    xbx^*,
    \qquad b\in D_d,\ x\in{\rm Mat}_{m,d}(\mathbb C).
\]

\begin{remark}
    These notions have the usual morphism-theoretic characterizations, with weak
    continuity entering on the minimal side. An operator system $\mathcal C$
    is $d$-maximal if and only if every $d$-positive map from an operator
    system into $\mathcal C$ is completely positive. Indeed, if
    $\varphi\colon\mathcal E\to\mathcal C$ is $d$-positive and
    $a\in E_m,$ then
    \[
        x^*({\rm id}_m\otimes\varphi)(a)x=({\rm id}_d\otimes\varphi)(x^*ax)\in C_d
        \qquad\text{for all }x\in{\rm Mat}_{m,d}(\mathbb C),
    \]
    so $d$-maximality implies $({\rm id}_m\otimes\varphi)(a)\in C_m$.

    Dually, $\mathcal D$ is $d$-minimal if and only if every weakly
    continuous $d$-positive map from $\mathcal D$ into an operator system
    admitting a dual is completely positive. In fact, weak continuity sends
    the weakly closed cone generated by the elements $xbx^*$ into the weakly
    closed positive cone of the target. The converses follow by applying
    these properties to the identity maps between the given system and the
    maximal, respectively weakly closed minimal, matrix ordering determined
    by its cone at level $d$.

    In both characterizations it is enough to test matrix algebras. More
    precisely, $\mathcal C$ is $d$-maximal if and only if every $d$-positive
    map ${\rm Mat}_n(\mathbb C)\to\mathcal C$ is completely positive, for
    every $n$, and $\mathcal D$ is $d$-minimal if and only if every weakly
    continuous $d$-positive map
    $\mathcal D\to{\rm Mat}_n(\mathbb C)$ is completely positive, for every
    $n$. To see the first assertion, suppose that
    $\varphi\colon\mathcal E\to\mathcal C$ is $d$-positive and let
    $a\in E_m$. By the usual Choi correspondence, $a$ corresponds to a completely positive map
    \(
        \psi_a\colon{\rm Mat}_m(\mathbb C)\longrightarrow\mathcal E.
    \)
    The composition
    $\varphi\circ\psi_a$ is $d$-positive and hence by assumption already completely positive. Its Choi matrix is
    \(({\rm id}_m\otimes\varphi)(a),
    \)
    which therefore belongs to $C_m$. Since this holds for every $m$ and
    every $a\in E_m$, the map $\varphi$ is completely positive. For the
    second assertion, if the
    weakly closed cone generated by the level-$d$ cone were strictly smaller
    than $D_m$, separation would produce a weakly continuous $d$-positive map
    $\mathcal D\to{\rm Mat}_m(\mathbb C)$ that is not completely positive.
\end{remark}

\begin{theorem}
    Assume $\mathcal C,\mathcal D$ are operator systems in duality. Then $\mathcal C$ is $d$-maximal if and only if $\mathcal D$ is $d$-minimal.
\end{theorem}
\begin{proof}
    For each $m$, let $K_m\subseteq {\rm Her}_m(W)$ be the weakly closed cone
    generated by all elements $xbx^*$ with
    $b\in D_d$ and $x\in{\rm Mat}_{m,d}(\mathbb C)$.  We thus have $K_m\subseteq D_m$.

    We now compute the dual cone of $K_m$.  For $a\in{\rm Her}_m(V)$,
    \[
        a\in K_m^\vee
        \quad\Longleftrightarrow\quad
        \langle a,xbx^*\rangle\geqslant 0
        \text{ for all }b\in D_d,
        \ x\in{\rm Mat}_{m,d}(\mathbb C).
    \]
    Using compatibility of the matrix pairings with scalar conjugation, this
    is equivalent to
    \[
        \langle x^*ax,b\rangle\geqslant 0
        \text{ for all }b\in D_d,
        \ x\in{\rm Mat}_{m,d}(\mathbb C),
    \]
    and since $C_d=D_d^\vee$, this is equivalent to
    \[
        x^*ax\in C_d
        \text{ for all }x\in{\rm Mat}_{m,d}(\mathbb C).
    \]
    Hence
    \[
        K_m^\vee
        =\{a\in{\rm Her}_m(V)\mid x^*ax\in C_d
        \text{ for all }x\in{\rm Mat}_{m,d}(\mathbb C)\}.
    \]
    If $\mathcal C$ is $d$-maximal, then this cone is $C_m$.  Thus
    $K_m^\vee=C_m=D_m^\vee$.  Since both $K_m$ and $D_m$ are weakly closed,
    bipolarity gives $K_m=D_m$.  Therefore $\mathcal D$ is $d$-minimal.

    Conversely, assume $\mathcal D$ is $d$-minimal, so $K_m=D_m$ for every
    $m$.  Then for $a\in{\rm Her}_m(V)$,
    \[
        a\in C_m
        \ \Leftrightarrow\
        a\in D_m^\vee
        \ \Leftrightarrow\
        a\in K_m^\vee
        \ \Leftrightarrow\
        x^*ax\in C_d
        \text{ for all }x\in{\rm Mat}_{m,d}(\mathbb C).
    \]
    Thus $\mathcal C$ is $d$-maximal.
\end{proof}

\begin{corollary}
    For $B(H)$ with $H$ separable, and for $C^*$-algebras with a faithful tracial state,  $d$-minimality and $d$-maximality (with respect to the above self-duality) are equivalent.
\end{corollary}

\subsection{Continuous and closed realizations}

Finally we discuss representation-theoretic consequences.  One should not
expect a weakly continuous concrete representation into a single $B(H)$ with
the norm topology, except in finite dimensions.  

\begin{lemma}
    Let $V,W$ be in duality and let $\rho\colon V\to B(H)_{\rm sa}$ be a linear map which is continuous from the weak topology to the usual operator norm on $B(H)$. Then $\rho$ has finite-dimensional range. In particular, if $\rho$ is injective, then $V$ is finite-dimensional.
\end{lemma}
\begin{proof}
    Since $\rho$ is norm-continuous at the origin, there are elements
    $w_1,\ldots,w_n\in W$ and $\varepsilon>0$ such that
    \[
        |\langle v,w_i\rangle|<\varepsilon\quad(1\leqslant i\leqslant n)
        \qquad\Longrightarrow\qquad
        \Vert\rho(v)\Vert<1.
    \]
    Let
    \[
        N=\{v\in V\mid \langle v,w_i\rangle=0\text{ for }1\leqslant i\leqslant n\}.
    \]
    If $v\in N$, then $tv$ belongs to the above weak neighbourhood for every
    $t>0$. Hence $t\Vert\rho(v)\Vert=\Vert\rho(tv)\Vert<1$ for all $t>0$, and
    therefore $\rho(v)=0$. Thus $N\subseteq\ker\rho$, so $\rho$ factors through
    the finite-dimensional quotient $V/N$. Consequently $\rho(V)$ is
    finite-dimensional. 
\end{proof}

The correct replacement is an $\ell^\infty$-product of finite-dimensional matrix algebras with its product topology.

\begin{proposition}
\label{thm:productreal}
    Let $\mathcal C$ be an operator system on $V$ that admits a dual. Then $\mathcal C$ admits a realization
    in an $\ell^\infty$-product of matrix algebras, which is continuous from
    the weak topology on $V$ to the product topology.
\end{proposition}

\begin{proof}
    The Choi--Effros construction uses positive linear functionals on
    ${\rm Her}_m(V)$.  Under the trace pairing, such a functional corresponds
    to a completely positive map $\psi\colon V\to{\rm Her}_m(\mathbb C)$.
    Compressing to the support of $\psi(e_V)$ and conjugating by its inverse
    square root makes the map unital.  Taking the product of all resulting
    unital completely positive maps gives a unital complete order embedding.

    If $\mathcal D$ is a dual of $\mathcal C$, the positive functionals needed
    to separate every cone $C_m$ are precisely the functionals represented by
    elements of $D_m=C_m^\vee$.  Every matrix entry of the corresponding map
    $\psi$ has the form $v\mapsto\langle v,w\rangle$ for some $w\in W$.
    Consequently each coordinate map is weak-to-norm continuous, and their
    product
    \[
        V\to\prod_i^{\ell^\infty}{\rm Her}_{m_i}(\mathbb C)
    \]
    is continuous from the weak topology to the product topology.
\end{proof}

The preceding realization has weak$^*$ closed range, under a natural Banach-Alaoglu compactness
assumption on the duality.

\begin{theorem}
\label{thm:weakstarreal}
    Let the operator systems $\mathcal C$ on $V$ and $\mathcal D$ on $W$ be in
    duality.  Assume that the order-unit norm unit ball of $V$ is compact for
    the weak topology induced by $W$.  Then there are a Hilbert space $H$ and
    a realization
    \[
        \rho\colon V\to B(H)_{\rm sa}
    \]
    whose range is a weak$^*$ closed operator subsystem of $B(H)$.
\end{theorem}

\begin{proof}
    By \Cref{thm:productreal}, there is a realization
    \[
        \rho\colon V\to
        M\coloneqq\prod_{i\in I}^{\ell^\infty}{\rm Her}_{m_i}(\mathbb C)
    \]
    which is continuous from the weak topology on $V$ to the product topology
    on $M$.  Since $\rho$ is a realization, it is an
    isometry for the order-unit norm.

    We claim that $\rho(V)$ is weak$^*$ closed in $M$, where $M$ is equipped
    with its canonical predual, the $\ell^1$-sum of the matrix spaces.  On
    each norm-bounded subset of $M$, the weak$^*$ topology agrees with the
    product topology.  Indeed, weak$^*$ convergence implies coordinatewise
    convergence.  Conversely, for a norm-bounded net, coordinatewise
    convergence implies convergence against every element of the $\ell^1$
    predual: One first makes the $\ell^1$-tail uniformly small and then uses
    convergence on the remaining finitely many coordinates.

    Let $B_V$ and $B_M$ denote the respective norm unit balls.  Since $\rho$
    is an isometry, for every $r>0$ we have
    \[
        \rho(V)\cap rB_M=\rho(rB_V).
    \]
    The set $rB_V$ is weakly compact by assumption, and $\rho$ is continuous
    from the weak topology to the product topology.  Since the product and
    weak$^*$ topologies agree on $rB_M$, the set $\rho(rB_V)$ is weak$^*$
    compact and hence weak$^*$ closed.  Thus the intersection of the linear
    subspace $\rho(V)$ with every closed norm ball is weak$^*$ closed.  The
    Krein--Smulian theorem now implies that $\rho(V)$ itself is weak$^*$
    closed in $M$.

    Finally, let
    \[
        N=\prod_{i\in I}^{\ell^\infty}{\rm Mat}_{m_i}(\mathbb C),
        \qquad M=N_{\rm sa}.
    \]
    The von Neumann algebra $N$ acts diagonally on the Hilbert space direct
    sum of the corresponding finite-dimensional spaces, and this
    representation is a weak$^*$ homeomorphism onto a weak$^*$ closed von
    Neumann subalgebra of $B(H)$.  Its restriction to $M$ therefore sends
    $\rho(V)$ onto a weak$^*$ closed operator subsystem of $B(H)$.
\end{proof}

\begin{remark}
    \Cref{thm:weakstarreal} is related to the setting from \cite{BlecherMagajna2011}.  They show that an
    operator system which is the dual of an operator space admits a completely
    isometric and weak$^*$ homeomorphic realization as a weak$^*$ closed
    operator subsystem of some $B(H)$.  Here no operator-space predual is
    assumed, but a dual in our sense, combined with compactness of the order-unit ball for the topology induced by the  dual system.
\end{remark}

\begin{proposition}
\label{prop:quotientcompact}
    Let the operator systems $\mathcal C$ on $V$ and $\mathcal D$ on $W$ be
    in duality, and assume that the order-unit norm unit ball $B_V$ is compact
    for the weak topology induced by $W$.  Let $K\subseteq V$ be weakly
    closed, assume that $e_W\in K^\perp$, and equip $X=V/K$ with the
    operator-system quotient from \Cref{thm:real}.  Suppose that $X$ is
    complete for its order-unit norm.
    Then the order-unit norm unit ball of $X$ is compact for the weak topology
    induced by the dual subsystem on $K^\perp$.  In particular,
    \Cref{thm:weakstarreal} applies to $X$.
\end{proposition}

\begin{proof}
    First observe that $V$ is complete for its order-unit norm.  Indeed, let
    $(v_n)$ be an order-norm Cauchy sequence.  It is contained in some multiple
    of $B_V$, and hence admits a subnet converging weakly to an element
    $v\in V$.  For every $\varepsilon>0$, the differences $v_n-v_m$ belong to
    $\varepsilon B_V$ whenever $n,m$ are sufficiently large.  Since
    $\varepsilon B_V$ is weakly closed, passing to the subnet limit shows that
    $v_n-v\in\varepsilon B_V$ for all sufficiently large $n$.  Thus $v_n$
    converges to $v$ in the order-unit norm.

    Let $q\colon V\to X$ be the quotient map.  By the discussion preceding
    \Cref{lem:dual}, the weak topology on $X$ induced by $K^\perp$ is the
    quotient topology, so $q$ is weakly continuous.  Moreover, $q$ is unital
    and positive, and hence contractive for the order-unit norms.  It is a
    surjective bounded linear map between Banach spaces, so the open mapping
    theorem provides a constant $c>0$ such that every $x\in X$ has a
    representative $v\in V$ with
    \[
        \lVert v\rVert_{\rm ou}\leqslant
        c\lVert x\rVert_{\rm ou}.
    \]
    Consequently,
    \[
        B_X\subseteq q(cB_V),
    \]
    where $B_X$ is the order-unit norm unit ball of $X$.  The set $q(cB_V)$
    is compact.  Moreover,
    \[
        B_X=\{x\in X\mid e_X\pm x\in E_1\},
    \]
    where $E_1$ is the positive cone of the quotient system.  This cone is
    weakly closed because it is the dual cone of the positive cone on
    $K^\perp$.  Hence $B_X$ is weakly closed, and is therefore a closed subset
    of the compact set $q(cB_V)$. Thus it is compact itself.
\end{proof}

\end{document}